\DeclareMathAlphabet{\mathpzc}{OT1}{pzc}{m}{it}

\documentclass[headsepline=true]{scrartcl}
\usepackage{amsmath}
\usepackage{amsthm, amssymb, amsfonts,bbm}
\usepackage[top=1.2in,bottom=1.2in,left=1in,right=1in]{geometry}
\usepackage[pdfborder={0 0 0}]{hyperref}
\usepackage{amscd}
\usepackage{tensor}
\usepackage{mathrsfs}
\usepackage{arydshln}
\usepackage{tikz}
\usetikzlibrary{decorations.markings}
\usetikzlibrary{decorations.pathreplacing}
\usepackage{lscape}
\usepackage{enumerate}
\usepackage[capitalize]{cleveref}

\title{On the modulo $p$ zeros of modular forms congruent to theta series}
\author{Berend Ringeling
	\thanks{This work is supported by NWO grant OCENW.KLEIN.006.}
	\\[1mm]
	\small Department of Mathematics, IMAPP, Radboud University,\\[-1mm] \small PO Box 9010, 6500~GL Nijmegen, Netherlands\\[0.5mm] \small \url{b.ringeling@math.ru.nl}
}

\newcommand{\sltwoz}{\mathrm{SL}_2(\z)}

\renewcommand{\phi}{\varphi}

\newcommand{\z}{\mathbb{Z}}
\newcommand{\q}{\mathbb{Q}}

\renewcommand{\c}{\mathbb{C}}

\theoremstyle{plain} 
\newtheorem{thm}{Theorem}[section]
\newtheorem{lem}[thm]{Lemma} 
\newtheorem{cor}[thm]{Corollary} 
\newtheorem{prop}[thm]{Proposition}

\theoremstyle{definition} 

\newtheorem{exmp}[thm]{Example}

\theoremstyle{remark}

\newcommand{\ii}{\mathrm{i}}

\renewcommand{\Im}{\mathrm{Im}}

\begin{document}
	\maketitle

\begin{abstract}
	For a prime  $p$ larger than $7$, the Eisenstein series of weight $p-1$ has some remarkable congruence properties modulo $p$. Those imply, for example, that the $j$-invariants of its zeros (which are known to be real algebraic numbers in the interval $[0,1728]$), are at most quadratic over the field with $p$ elements and are congruent modulo $p$ to the zeros of a certain truncated hypergeometric series.
	In this paper we introduce ``theta modular forms'' of weight $k \geq 4$ for the full modular group as the modular forms for which the first $\dim(M_k)$ Fourier coefficients are identical to certain theta series. We consider these theta modular forms for both the Jacobi theta series and the theta series of the hexagonal lattice. 	
	We show that the $j$-invariant of the zeros of the theta modular forms for the Jacobi theta series are modulo $p$ all in the ground field with $p$ elements. For the theta modular form of the hexagonal lattice we show that its zeros are at most quadratic over the ground field with $p$ elements.
	Furthermore, we show that these zeros in both cases are congruent to the zeros of certain truncated hypergeometric functions.
\end{abstract}
	
	\section{Introduction}
		For $k \in \z_{\geq 0}$ and a congruence subgroup $\Gamma \subset \sltwoz$, denote by $M_k(\Gamma)$ the $\q$-vector space of modular forms with rational $q$-expansion of weight $k$ for $\Gamma$.
		 If $\Gamma = \text{SL}_2(\mathbb{Z})$ we simply write $M_k$ for $M_k(\text{SL}_2(\mathbb{Z}))$. 
		 Basic examples of modular forms include the \emph{Eisenstein series of weight $k$}, given by the $q$-expansion
	\begin{equation*}
		\label{Eisensteins}
		E_k(\tau) = 1-\frac{2k}{B_k}\sum_{n = 1}^{\infty} \left(\sum_{d|n} d^{k-1} \right) q^n \in M_k \qquad (q = e^{2\pi \ii \tau})
	\end{equation*}
	for even $k \geq 4$, where $B_k$ is the $k$-th Bernoulli number and $\tau \in \mathbb{H} = \{ \tau \in \c \, | \, \Im \, \tau > 0\}$.
	Another example is the \emph{modular discriminant} \begin{equation*}
		\Delta(\tau) = q\prod_{n = 1}^{\infty}(1 - q^n)^{24}  = \frac{1}{1728}(E_4(\tau)^3 - E_6(\tau)^2) \in M_{12}.
	\end{equation*}
	It is well known that the space $M_k$ is finite-dimensional: writing $k \geq 4$ uniquely as
	\begin{equation}
		\label{indicesnot}
		k = 12n_k + 4a_k + 6b_k , \quad \text{where } n_k \in \mathbb{Z}_{\geq 0}, \, a_k \in \{0, 1,2 \}, \, b_k \in \{0,1\},
	\end{equation}
	an explicit basis for $M_k$ is given by
	\begin{equation}
		\label{basis}
		\{ \Delta^{n_k - \ell} E_4^{a_k + 3\ell}E_6^{b_k} \, | \, 0 \leq \ell \leq n_k \}.
	\end{equation}
Furthermore, define the \emph{modular $j$-invariant} as
\begin{equation*}
	j(\tau) = \frac{E_4(\tau)^3}{\Delta(\tau)};
\end{equation*}
it is a weakly holomorphic modular form (poles at the cusps are allowed) of weight $0$.
	For any $f \in M_k$, using the notation in \eqref{indicesnot}, consider the quotient
	\begin{equation*}
		Q[f] = \frac{f}{\Delta^{n_k}E_4^{a_k}E_6^{b_k}};
	\end{equation*}
this is a weakly holomorphic modular form of weight $0$.
	It follows from \eqref{basis} that there exists a polynomial $P[f](j) \in \q[j]$ of degree $\leq n_k$ such that  
	\begin{equation*}
		Q[f](\tau) = P[f](j(\tau)).
	\end{equation*}
	Explicitly, these polynomials can be written as
	\begin{equation*}
		P[f](j) = j^{\frac{\textnormal{ord}_{\rho}(f) - a_k}{3}}(j-1728)^{\frac{\textnormal{ord}_{\ii}(f) - b_k}{2}} \prod_{\substack{\tau \in \mathcal{F}, f(\tau) = 0 \\ j(\tau) \neq 0, 1728}} (j - j(\tau)),
	\end{equation*}
where $\mathcal{F} \subset \mathbb{H}$ denotes the standard fundamental domain and where $\ii$ and $\rho = e^{2 \pi \ii/3}$ are the elliptic points of $\mathcal{F}$.
Thus, the zeros of the polynomial $P[f]$ contain the zeros of $f$ as input; naturally this polynomial plays an important role in the study of zeros of modular forms.
	For example, one can consider the polynomial $P[f]$ for $f = E_k$.
	Since the orders at the elliptic points are given by $\textnormal{ord}_{\rho}(E_k) = a_k$ and $\textnormal{ord}_{\ii}(E_k) = b_k$, see \cite{Rankin}, it follows that
	\begin{equation*}
		P[E_k](j) = \prod_{\substack{\tau \in \mathbb{H}, \, E_k(\tau) = 0\\ j(\tau) \neq 0, 1728}} (j - j(\tau)).
	\end{equation*}
In 1970, it was shown by F.K.C. Rankin and H.P.F. Swinnerton-Dyer \cite{Rankin} that the non-elliptic zeros of $E_k(\tau)$ in $\mathcal{F}$ are all simple and located on the arc \begin{equation}
	\label{arc}
	A = \Bigl \{ e^{\ii \alpha}\ | \, \frac{\pi}{2} < \alpha < \frac{2 \pi}{3} \Bigr \}.\end{equation}Since the $j$-image of the arc is $j(A) = (0, 1728)$, it follows that the zeros of $P[E_k](j)$ are all simple and located in the real interval $(0,1728)$. Moreover, $P[E_k](j) \in \mathbb{Q}[j]$ implies that the $j$-invariants of the zeros are all algebraic numbers. Another feature of this polynomial is the location of the zeros over finite fields. For primes $p \geq 5$, the $q$-series coefficients of $E_{p-1}$ are all $p$-integral, so that the coefficients of the polynomials $P[E_{p-1}](j)$ are also all $p$-integral. Hence these polynomials can all be reduced modulo $p$. We have the following surprising result for them.
\begin{thm}[Deligne  {{\cite[Section 2.1]{Katz}}}]
	\label{A}
	For primes $p \geq 5$, we have the congruence
	\begin{equation}
		P[E_{p-1}](j) \equiv \prod (j - j(\mathcal{E})) \mod p,
	\end{equation}
	where the product ranges over all supersingular elliptic curves $\mathcal{E}/\overline{\mathbb{F}}_p$, up to isomorphism, with $j$-invariant $j(\mathcal{E}) \neq 0, 1728$.
\end{thm}
In fact, it was shown by M. Deuring \cite{Deuring}
 that $j(\mathcal{E})$ lies in $ \mathbb{F}_{p^2}$ for supersingular elliptic curves $\mathcal{E}$ over $\overline{\mathbb{F}}_p$. Therefore, we have the following theorem.
\begin{thm}[\cite{Deuring},\cite{Katz}]
	\label{B}
	For primes $p \geq 5$,
	$P[E_{p-1}](j)$ factors as a product of linear and quadratic factors over $\mathbb{F}_p$.
\end{thm}
The number of linear factors of this polynomial is related \cite{Brillhart} to the class number of the field $\mathbb{Q}(\sqrt{-p})$. More arithmetic properties for these polynomials can be found for example in \cite{Gekeler}. Natural extensions of these polynomials to rational function fields have been studied in \cite{Cornelissen}. 

The polynomials $P[E_{p-1}](j) \mod p$ can also be cast as truncated hypergeometric functions. We first introduce some notation.
For $x \in \mathbb{C}$ and $n \in \mathbb{Z}_{\leq 0}$, define the \emph{Pochhammer symbol} as $$(x)_n = \begin{cases} 1 &\textnormal{ if }n = 0, \\ x(x+1) \cdots (x + n - 1)&\textnormal{ if }n > 0.\end{cases}$$
For $\alpha, \beta \in \mathbb{C}$ and $\gamma \in \mathbb{C} \setminus \mathbb{Z}_{\leq 0}$ a \emph{hypergeometric function} is
	\begin{equation*}
		 {}_{2} F_{1} \left(\alpha, \beta; \gamma; z\right) := \sum_{n \geq 0} \frac{(\alpha)_n (\beta)_n}{(\gamma)_n} \frac{z^n}{n!},
	\end{equation*}
the series defines an absolutely convergent series in the disk $|z| < 1$. Finally, define the polynomials $U^{0}_n(j)$ and $U^{1}_n(j)$ for $n \geq 0$ as the unique polynomials of degree $n$ satisfying
	\begin{align}
		j^n \cdot {}_{2} F_{1}\left(\frac{1}{12},\frac{5}{12};1;\frac{1728}{j}\right) &= j^n \left(1 + \frac{60}{j} + \frac{39780}{j^2} + \cdots \right) = U_n^{0}(j) + \mathcal{O}(1/j)\label{kn1},\\
		j^n \cdot {}_{2} F_{1}\left(\frac{7}{12},\frac{11}{12};1;\frac{1728}{j}\right) &= j^n \left(1 + \frac{924}{j} + \frac{1211364}{j^2} + 
		\cdots \right) = U_n^{1}(j) + \mathcal{O}(1/j).\label{kn2}
	\end{align}
These truncations of hypergeometric functions satisfy the following congruences.
	\begin{thm}[{Kaneko-Zagier \cite[Proposition 5]{KZ}}]
		\label{C}
		If $k = p - 1$, we have the congruence
		$$P[E_k](j) \equiv U^{b_k}_{n_k}(j) \mod p.$$
	\end{thm}
The main goal of this paper is to prove analogues of Theorems \ref{A}, \ref{B} and \ref{C} for a different modular setup that we outline in the next section.
	\subsection{Theta modular forms}	
Given a formal power series $f(q) \in \c[[q]]$ and an even integer $k \geq 4$, there is a unique modular form $\mathscr{C}_k f \in M_k(\sltwoz)$
	such that 
	\begin{equation}
		\label{constr}
		\mathscr{C}_k f - f = \mathcal{O}(q^{n_k+1}),
	\end{equation}
where $n_k$ is defined in \eqref{indicesnot}.
Example \ref{exmp1} below shows that there is indeed a natural construction of this form. We first introduce some notation.
For  a prime $p$ and $f,g \in \mathbb{Q}[[q]]$ with $p$-integral coefficients write
\[ f \equiv g + \mathcal{O}(q^m) \mod p\]
for some $m \geq 0$ if the first $m$ Fourier coefficients of $f$ and $g$ agree modulo $p$. Furthermore, write $f \equiv g \mod p$ if all Fourier coefficients agree modulo $p$.
\begin{exmp}
	\label{exmp1}
	Consider the formal power series $f = 1 \in \c[[q]]$. The resulting modular forms $\mathscr{C}_k 1$ are called \emph{extremal modular forms} and are related to the theory of extremal lattices, see for example \cite{Jenkins}. In 2007, W. Duke and P. Jenkins \cite{DukeJenkins} showed that the non-elliptic zeros of $\mathscr{C}_k 1$ are all simple and located on the arc $A$ inside $\mathcal{F}$, see \eqref{arc}, just as in the case of $E_k$. Thus, we see that the polynomials $P[\mathscr{C}_{k} 1](j)$ and $P[E_{k}](j)$ have similar factorisation behaviour in the ring $\mathbb{R}[j]$. But this is also the case in the ring $\mathbb{F}_p[j]$: indeed, one can study congruence properties of these modular forms.
	For primes $p \geq 5$ the congruence of the Bernoulli numbers \[ \frac{1}{B_{p-1}} \equiv 0 \mod p\] imply the congruence of power series \[E_{p-1} \equiv 1 \mod p,\]
	so that
	\[\mathscr{C}_{k} 1 \equiv E_{k} + \mathcal{O}(q^{n_k+1}) \mod p\]
	when $k = p-1$.
	We will see later in Lemma \ref{almostzeroiszero} that this implies the congruence $$P[\mathscr{C}_{k} 1](j) \equiv P[E_{k}] (j) \mod p.$$
\end{exmp}
We now introduce \emph{theta modular forms}. For a positive definite lattice $L$, define its theta series as
\[ \theta_L(\tau) = \sum_{x \in L} q^{||x||^2}.\] This is a holomorphic function on $\mathbb{H}$. We call $\mathscr{C}_k \theta_L(\tau)$ the \emph{theta modular form} of weight $k$ corresponding to the lattice $L$. In this paper we will consider the (one-dimensional) lattice $\mathbb{Z}$ and the (hexagonal) lattice $H = \mathbb{Z}(1,0) + \mathbb{Z}(\frac{1}{2}, \frac{1}{2}\sqrt{3})$.
Their theta series are
\begin{align*}
	\theta_\mathbb{Z} &= \sum_{n \in \mathbb{Z}} q^{n^2} = 1 + 2q + 2q^4 + 2q^9 + \cdots \,
\end{align*}
known as the \emph{Jacobi theta series} and
\begin{align*}
		\qquad \qquad \theta_H &= \sum_{m,n \in \mathbb{Z}} q^{m^2 + n^2 + mn} = 1 + 6q + 6q^3 + 6q^4 + \cdots \, .
\end{align*}
For these theta series it is known that $\theta_\mathbb{Z}^2 \in M_1(\Gamma_1(4))$ and $\theta_H \in M_1(\Gamma_1(3))$, see for example \cite{Cooper}.

\begin{exmp}
\label{k52}
For $k = 52$, the theta modular form corresponding to the lattice $\z$ is
	\begin{align*}
		\mathscr{C}_{52}\theta_{\z} &= 27800506386 E_4\Delta^4 - 776608440 E_4^4\Delta^3 +2887488 E_4^7\Delta^2 - 3118E_4^{10}\Delta + E_4^{13}\\
		&= 1 + 2q + 2q^2 + 2q^4 + 95037348924q^5 + 1017845969208768q^6 + \cdots \, .
	\end{align*}
\end{exmp}

In this paper we find analogues of Theorems \ref{A}, \ref{B} and \ref{C} for the modular forms $\mathscr{C}_k \theta_{\z}$ and $\mathscr{C}_k\theta_H$. Namely, we will write the polynomials $P[\mathscr{C}_{k}\theta_{\z}]$ and $P[\mathscr{C}_{k}\theta_{H}]$ as truncated hypergeometric functions modulo primes, similar to what is done for the Eisenstein series in Theorem \ref{C}, see  Theorems \ref{MainTHMA} and \ref{HexP} below. Furthermore, we will see that the factorisations of the polynomials $P[\mathscr{C}_{k}\theta_{\z}]$ and $P[\mathscr{C}_{k}\theta_{H}]$ over finite fields have a structure reminiscent to that of $P[E_k]$, as recorded in Theorems \ref{A} and \ref{B}. Over finite fields, $P[\mathscr{C}_{k}\theta_{\z}]$ factors as a product of linear factors, see Theorem \ref{MainTHMB}, while $P[\mathscr{C}_{k}\theta_{H}]$ factors as a product of quadratic factors only (and one linear factor if the degree of the polynomial is odd), see Theorem \ref{HexF}, whereas $P[E_k]$ factors as the product of both linear and quadratic factors.

\subsection{Results}
	
	In this section we state our results about factorisations of the polynomials $P[\mathscr{C}_{k}\theta_{\z}](j)$ and $P[\mathscr{C}_{k}\theta_{H}](j)$ over finite fields. Since $\mathscr{C}_k \theta_{\z}(\tau)$ and $\mathscr{C}_k \theta_H(\tau)$ have integer Fourier coefficients, the polynomials $P[\mathscr{C}_k \theta_{\z}](j)$ and $P[\mathscr{C}_k \theta_H](j)$ have integer coefficients as well. This means that reduction modulo primes is always well-defined.
	\subsubsection{Results for $\mathscr{C}_k \theta_{\z}(\tau)$}
	Define the polynomials $W_n^{0}(j)$ and $W_n^{1}(j)$ for $n \geq 0$ as the unique polynomials of degree $n$ satisfying
	\begin{align}
		j^n \cdot {}_{2} F_{1}\left(\frac{-1}{24},\frac{7}{24};\frac{3}{4};\frac{1728}{j}\right) &= j^n \cdot \left(1 - \frac{28}{j} - \frac{17112}{j^2} + \cdots \right) = W_n^{0}(j) + \mathcal{O}(1/j), \label{nr1}\\
		j^n \cdot {}_{2} F_{1}\left(\frac{11}{24},\frac{19}{24};\frac{3}{4};\frac{1728}{j}\right) &= j^n \cdot \left(1 + \frac{836}{j} + \frac{1078440}{j^2} + \cdots \right)  = W_n^{1}(j) + \mathcal{O}(1/j)\label{nr2}.
	\end{align}
	For these polynomials we have the following congruences.
	\begin{thm}
		\label{MainTHMA}
		Let $p \geq 7$ be a prime and $k = \frac{p+1}{2}$. Then
		\begin{equation}
			P[\mathscr{C}_k \theta_{\z}](j) \equiv W_{n_k}^{b_k}(j) \mod p,
		\end{equation}
where $b_k$ and $n_k$ are defined in \eqref{indicesnot}.
	\end{thm}
	Note how the parameters in Equations \eqref{nr1} and \eqref{nr2} are approximately halved compared to \eqref{kn1} and \eqref{kn2}.
	Using this hypergeometric expression for $P[\mathscr{C}_k \theta_{\z}](j) \mod p$, we will show that this polynomial splits over the ground field $\mathbb{F}_p$.
	\begin{thm}
		\label{MainTHMB}
		Let $p \geq 7$ be a prime and $k = \frac{p+1}{2}$. Then $P[\mathscr{C}_k \theta_{\z}](j)$ splits over $\mathbb{F}_p$ into linear factors.
	\end{thm}
	Thus, Theorems \ref{MainTHMA} and \ref{MainTHMB} are analogues of Theorems \ref{C} and \ref{B} respectively.
	 The modular forms $\mathscr{C}_k \theta_{\z}$ for $k = \frac{p+1}{2}$ are congruent, up to a constant, to the weight $k$ modular forms $Th(\phi^{H_1}_{2k})$ defined by T. Miezaki in \cite{Miezaki}. Here $\phi^{H_1}_{2k}$ are certain invariant polynomials related to the genus $1$ average weight enumerator of binary, self-dual and doubly even codes of length $2 k$, see \cite{Oura}.
	 
	
	Explicitly, the zero set of $P[\mathscr{C}_k \theta_{\z}](j) \mod p$ is given by
	\begin{equation}
		\left  \{ \frac{256(1 - \lambda + \lambda^2)^3}{\lambda^2 (\lambda - 1)^2} \, \Bigm |  -\lambda, \lambda - 1 \in \mathbb{F}_p^{*2} \right \} \setminus \{0, 1728 \}.
	\end{equation}
	This zero set has an interpretation via elliptic curves. Let $p \geq 7$ be a prime and $\mathcal{E}$ an elliptic curve over $\mathbb{F}_p$ given by the equation
	\begin{equation}
		\mathcal{E}: Y^2 = X^3 + aX + b,
	\end{equation}
	where $a,b \in \mathbb{F}_p$, with $4a^3 + 27b^2 \not \equiv 0 \mod p$.
	Denote by $\mathcal{E}(\mathbb{F}_p)$ the group of $\mathbb{F}_p$-rational points on $\mathcal{E}$ and by $\mathcal{E}(\mathbb{F}_p)[n]$ the $n$-torsion subgroup of $\mathcal{E}(\mathbb{F}_p)$, where $n \in \mathbb{N}$. It is well known that $\mathcal{E}(\mathbb{F}_p)[n]$ is isomorphic to a subgroup of $\mathbb{Z}/n\mathbb{Z} \times \mathbb{Z}/n\mathbb{Z}$ \cite[Corollary 6.4]{Silverman}.
	We have the following result.
	\begin{thm}
		\label{MainTHMC}
		For a prime $p \geq 7$ and $k = \frac{p+1}{2}$, the congruence
		$$P[\mathscr{C}_k \theta_{\z}](j) \equiv \prod_{\substack{\mathcal{E}/\mathbb{F}_p \text{\large $/$} \cong \\ |\mathcal{E}(\mathbb{F}_p)[2]| = |\mathcal{E}(\mathbb{F}_p)[4]| = 4 \\ j(\mathcal{E}) \neq 0, 1728 }} (j -  j(\mathcal{E})) \quad \mod p$$
		holds. The product here is over elliptic curves defined over $\mathbb{F}_p$, up to $\overline{\mathbb{F}}_p$-isomorphism, with full rational $2$-torsion and no rational points of order $4$.
	\end{thm}

	\begin{exmp}
		For $p=103$, we have $k = 52, a_k = 1$ and $b_k = 0$. Continuing the computation in Example \ref{k52} we obtain $$P[\mathscr{C}_{52} \theta_{\z}](j) = \frac{\mathscr{C}_{52} \theta_{\z}}{E_4\Delta^4} = j^4 - 3118 j^3 + 2887488 j^2 - 776608440 j + 27800506386.$$
		Now $$P[\mathscr{C}_{52} \theta_{\z}](j) \equiv (j - 58)(j - 89)(j - 93)(j - 97)  \mod p$$
		and $$P[\mathscr{C}_{52} \theta_{\z}](j) \equiv j^4 - 28j^3 - 17112j^2  - 16085280j + 18044467104 \equiv W^{0}_{4}(j) \mod p.$$
		Computing the $j$-invariants of the elliptic curves
		$$\mathcal{E}: \ Y^2 = X^3 + aX + b, \quad a,b \in \mathbb{F}_p,$$
		with $\mathcal{E}(\mathbb{F}_p)[2] = \mathcal{E}(\mathbb{F}_p)[4] \cong \mathbb{Z}/2\mathbb{Z} \times \mathbb{Z}/2\mathbb{Z}$
		we indeed find $j(\mathcal{E}) \in \{0, 58, 89, 93, 97 \}$, in agreement with the statement in Theorem \ref{MainTHMC}.
	\end{exmp}
	\subsubsection{Results for $\mathscr{C}_k \theta_H(\tau)$}
	Define the polynomials $V^{0}_n(j)$ and $V^{1}_n(j)$ for $n \geq 0$ as the 
	unique polynomials of degree $n$ satisfying
	\begin{align*}
		j^n \cdot {}_{2} F_{1}\left(-\frac{1}{12},\frac{1}{4};\frac{2}{3};\frac{1728}{j}\right) &= j^n \cdot \left(1 - \frac{54}{j} - \frac{32076}{j^2} + \cdots \right) = V^{0}_n(j) + \mathcal{O}(1/j),\\
		j^n \cdot {}_{2} F_{1}\left(\frac{5}{12},\frac{3}{4};\frac{2}{3};\frac{1728}{j}\right) &= j^n \cdot \left(1 + \frac{810}{j} + \frac{1041012}{j^2} + \cdots \right) = V^{1}_n(j) + \mathcal{O}(1/j).
	\end{align*}
	We have the following congruences for them.
	\begin{thm}
		\label{HexP}
		Let $p$ be a prime congruent to $5$ or $11$ modulo $12$ and $k = p + 1$. Then we have the congruence
		\begin{equation*}
			P[\mathscr{C}_k \theta_H](j) \equiv V^{b_k}_{n_k}(j) \mod p,
		\end{equation*}
for the choice of $b_k$ and $n_k$ as in \eqref{indicesnot}.
	\end{thm}
	Furthermore, the polynomials $P[\mathscr{C}_k \theta_H](j)$ have a specific factorisation modulo $p$.
	\begin{thm}
		\label{HexF}
		Let $p$ be a prime congruent to $5$ or $11$ modulo $12$ and $k = p + 1$. The polynomials $P[\mathscr{C}_k \theta_H](j)$ split over $\mathbb{F}_{p^2}$. Moreover, these polynomials factor over $\mathbb{F}_p$ as a product of quadratic factors only if the degree $n_k$ is even and as $(j + 1728)$ times a product of quadratic factors if the degree $n_k$ is odd.
	\end{thm}
Explicitly, its zero set is
\begin{equation}
	\label{ZeroHexa}	\left\{ \frac{3^3 4^4 (2 a -1)^3}{a(a+4)^3} \, \Bigm | \,  a^{(p+1)/3} + 2^{1/3} = 0, \, a \in \mathbb{F}_{p^2}  \right\} \setminus \{0,1728\},
\end{equation}
where $2^{1/3}$ is the unique cube root of $2$ in $\mathbb{F}_p$.

There is no clear analogy with Theorem \ref{MainTHMC} for the polynomials $P[\mathscr{C}_k \theta_H](j)$. However, by an explicit calculation we get the following characterisation of the zero set \eqref{ZeroHexa}.
\begin{prop}
	The zero set \eqref{ZeroHexa} coincides with the set of $j$-invariants of the elliptic curves
	\begin{equation}
		\mathcal{E}_b: X^3 + Y^3 + 1 = 3 b X Y,
	\end{equation}
 whenever it is non-singular, with
	$b \in \mathbb{F}_{p^2}$ satisfying $b^{p+1} \equiv -2 \mod p$. 
\end{prop}
The curves $\mathcal{E}_b$, known as \emph{Hessian curves}, have full rational $3$-torsion \cite[Theorem 10]{Moody}, $\mathcal{E}_b(\mathbb{F}_{p^2})[3] \cong \mathbb{Z}/3 \mathbb{Z} \times \mathbb{Z}/3 \mathbb{Z}$, indicating an analogy with Theorem \ref{MainTHMC}. However, the condition $b^{p+1} \equiv -2 \mod p$ \, seems to lack a good interpretation. 
\begin{exmp}
		Take $p = 107$. In this case $k = 108$ and we find out that 
		\[ \mathscr{C}_{108} \theta_H = 1 + 6q + 6q^3 + 6q^4 + 12q^7 + 6q^9 + 1496265431568669020160q^{10} + \cdots \]
		and 
		\begin{align*} P[\mathscr{C}_{108} \theta_H](j) = \frac{\mathscr{C}_{108} \theta_H}{\Delta^{9}} &= j^9 - 6474 \, j^8 + 16858944 \, j^7 - 22595806434 \, j^6 + 16561497291750 \, j^5 \\& \qquad - 6514224685621164 \, j^4 + 1257337803035458656 \, j^3 \\& \qquad- 97749420668058422880 \, j^2 + 1958195577341989938240 \, j \\& \qquad - 2139590870258478384000.
		\end{align*}
		We check the congruence properties in Theorems \ref{HexF} and \ref{HexP}.
		We see that
		\[P[\mathscr{C}_{108} \theta_H](j) \equiv j^9 - 54 j^8 - 32076 j^7 + \cdots \equiv V_{9}^{0}(j) \mod p\]
		and 
		\[ P[\mathscr{C}_{108} \theta_H](j) \equiv (j + 16)(j^2 + 42)(j^2 + 6j + 42)(j^2 + 33j + 42)(j^2 + 105j + 42) \mod p,
		\]
		so that $P[\mathscr{C}_{108} \theta_H](j)$ factors as $(j+1728)$ times quadratic factors over $\mathbb{F}_p$.
	\end{exmp}
	
	
	\section{Proofs}
	We start with the following basic observation.
	\begin{lem}
		\label{almostzeroiszero}
		Let $k \geq 4$ and $p \geq 5$ a prime.
		Suppose $f \in M_k$ has $p$-integral Fourier coefficients and
		\begin{equation*}
			f \equiv \mathcal{O}(q^{n_k + 1}) \mod p \, 
		\end{equation*} (that is, the first $n_k$ Fourier coefficients of $f$ are divisible by $p$). Then $f \equiv 0 \mod p$.
	\end{lem}
	\begin{proof}
		Suppose $f$ satisfies the conditions of the lemma.
		Note that the $\mathbb{F}_p$-vector space $\{g \mod p \ | \ g \in M_k, \ g \textnormal{ is $p$-integral} \}$ is $(n_k + 1)$-dimensional. It is easy to see that
		\begin{equation}
			\{ \Delta^{n_k - \ell} E_4^{a_k + 3\ell}E_6^{b_k} \mod p \, | \, 0 \leq \ell \leq n_k \}
		\end{equation}
		is a basis for this space, thus $f \equiv 0 \mod p$.
	\end{proof}
Given two modular forms $f,g \in M_k$ with $p$-integral Fourier coefficients, Lemma \ref{almostzeroiszero} implies that they agree modulo $p$ if their first $n_k+1$ Fourier coefficients coincide modulo $p$.
	
	\subsection{Proofs for $\mathscr{C}_k \theta_{\z}$}
	We start with a hypergeometric identity for the function $\theta_{\z}(\tau)$.
	\begin{lem}
		In a neighborhood of $\tau = \ii \infty$, we have
		\label{lem22}
		\begin{equation}
			\label{DEGEsq}
			\theta_{\z}(\tau) = E_4(\tau)^{1/8} \, {}_2 F_{1} \left(\frac{-1}{24},\frac{7}{24};\frac{3}{4};\frac{1728}{j(\tau)}\right).
		\end{equation} 
	\end{lem}
	\begin{proof}
		The theta series $\theta_{\z}(\tau)^2$ is a modular form of weight $1$ for the congruence subgroup $\Gamma_1(4)$. Therefore, by \cite[Proposition 21]{Zagbook}, it satisfies a second order linear differential equation as a function in $1728/j$. 
		 As \begin{equation}
			\label{E4HG}
			E_4(\tau)^{1/4} = {}_2 F_{1} \left(\frac{1}{12},\frac{5}{12};1;\frac{1728}{j(\tau)}\right),
		\end{equation} see \cite[Eq. (74)]{Zagbook}, and the ${}_2 F_{1}$ satisfies a second order linear differential equation, it is easy to check that the left-hand side and the right-hand side of \eqref{DEGEsq} satisfy the same differential equation and initial values.
	\end{proof}

	\begin{proof}[Proof of Theorem \ref{MainTHMA}]
		Let $p \geq 7$ be a prime and $k = \frac{p+1}{2}$. Consider the modular form
		\[ h_k(\tau) =  W_{n_k}^{b_k} (j(\tau))\Delta(\tau)^{n_k}E_4(\tau)^{a_k} E_6(\tau)^{b_k} \]
		of weight $k$.
		In order to show that $P[\mathscr{C}_k \theta_{\z}](j)$ and $ W_{n_k}^{b_k}(j)$ agree modulo $p$ as polynomials in $j$, it suffices to show the congruence 
		\begin{equation}
			\label{DeEerste}
			h_k \equiv \mathscr{C}_k \theta_{\z} \mod p 
		\end{equation}
		for power series in $q$.
		 We first assume $k \equiv 0, 4\mod 12$, i.e. $b_k = 0$. Using Lemma \ref{lem22} we find
		\begin{align*}
		h_k &= \left(j^{n_k} {}_{2} F_{1}\left(-\frac{1}{24},\frac{7}{24};\frac{3}{4};\frac{1728}{j}\right) + \mathcal{O}(1/j) \right)\Delta^{n_k}E_4^{a_k} \\
			&= \theta_{\z}E_4^{3n_k + a_k - 1/8} + \mathcal{O}(q^{n_k + 1}).
		\end{align*}
		As $3n_k + a_k - 1/8 = p/8$, we see that $$E_4^{3n_k + a_k - 1/8} =  E_4^{p/8} \equiv 1 + \mathcal{O}(q^{p}) \mod p,$$
		therefore
		\[h_k \equiv \mathscr{C}_k \theta_{\z} + \mathcal{O}(q^{n_k+1}) \mod p. \] 
		From Lemma \ref{almostzeroiszero} we conclude that $h_k \equiv \mathscr{C}_k \theta_{\z} \mod p$.
		\newline
		Now assume $k \equiv 6, 10 \mod 12$, i.e. $b_k = 1$. Using Euler's transformation formula for hypergeometric functions \cite[Theorem 2.2.5]{Andrews} we find
		\begin{align*}
			{}_{2} F_{1}\left(\frac{11}{24},\frac{19}{24};\frac{3}{4};\frac{1728}{j}\right) &= \left( 1 - \frac{1728}{j} \right)^{-1/2}     {}_{2} F_{1}\left(\frac{-1}{24},\frac{7}{24};\frac{3}{4};\frac{1728}{j}\right)\\
			&= E_4^{3/2}E_6^{-1} {}_{2} F_{1}\left(\frac{-1}{24},\frac{7}{24};\frac{3}{4};\frac{1728}{j}\right).
		\end{align*}
		Hence
		\begin{align*}
			h_k(\tau) &= \left( j^{n_k} {}_{2} F_{1}\left(\frac{11}{24},\frac{19}{24};\frac{3}{4};\frac{1728}{j}\right) + \mathcal{O}(1/j)  \right) \Delta^{n_k} E_4^{a_k}E_6\\
			&= \theta_{\z}E_4^{3n_k + a_k + 11/8} + \mathcal{O}(q^{n_k + 1}).
		\end{align*}
		As before we have $3n_k + a_k + 11/8 = p/8$, so that 
		\begin{equation}
			h_k \equiv \mathscr{C}_k \theta_{\z} + \mathcal{O}(q^{n_k+1}) \mod p.
		\end{equation}
		Again, Lemma \ref{almostzeroiszero} implies $h_k \equiv \mathscr{C}_k \theta_{\z} \mod p$.
		Thus $P[\mathscr{C}_k \theta_{\z}](j) \equiv W_{b_k}^{n_k}(j)$ for all $p \geq 7$.
	\end{proof}
 
 	We will now work towards the proof of Theorem \ref{MainTHMB}. The goal is to make a choice of a Hauptmodul $t_2(\tau)$ for the group $\Gamma(2)$ and write $j(\tau)$ as a rational function in $t_2(\tau)$. This induces a variable transformation for the polynomial $P[\mathscr{C}_k \theta_{\z}](j)$ 
 	that simplifies the treatment of the zeros. Consider the \emph{modular lambda function} 
 	\[\lambda(\tau) = 16 \left( \frac{\eta(\tau) \eta^2(4 \tau)}{\eta^3(2 \tau)} \right)^8 = 16q^{1/2} - 128q + 704 q^{3/2} - 3072 q^2 + 11488 q^{5/2} + \cdots,\]
 	where $$\eta(\tau) = q^{1/24} \prod_{n = 1}^\infty (1- q^n) \, $$ 
 	is the \emph{Dedekind eta function}. 	 Choose $\lambda(\tau)$ as Hauptmodul $t_2(\tau)$ for the group $\Gamma(2)$ and write $j(\tau)$ as a rational function in $t_2(\tau)$:
 	\[j(\tau) =  \frac{256(1 - t_2(\tau) + t_2(\tau)^2)^3}{t_2(\tau)^2(t_2(\tau) - 1)^2}.\]

 	Using Lemmas \ref{lem25} and \ref{lem27} below, we will write the zeros of $P[\mathscr{C}_k \theta_{\z}] \mod p$ in terms of $t_2$.
 	
 	We first start with some technical statements.
	\begin{lem}
		\label{lem13}
		For $n \in \mathbb{Z}_{\geq 0}$ let $p$ be a prime $p \in \{24n - 1, 24n + 7 \}$.  For $m \geq 0$ denote by $c_m \in \mathbb{Q}$ the $m$-th coefficient in the expansion of ${}_2 F_{1}\left(-\frac{1}{24},\frac{7}{24};\frac{3}{4};x\right)$. Then $c_m \equiv 0 \mod p$ for $n < m < 6n$. 
	\end{lem}
	\begin{proof}
		Consider the prime of the form $p = 24n -1$.
		By considering the $p$-adic valuation $\nu _p$ of the coefficients,
		we find $\nu _p ( (-\frac{1}{24})_m) \geq 1$ if and only if $m > n$ and $\nu _p ((
		\frac{3}{4})_m) \geq 1$ if and only if $m \geq 6n$. The case of $p=24n+7$ is similar.
	\end{proof}
	
	\begin{lem}
		\label{lem14}
		For $n \in \mathbb{Z}_{\geq 0}$ let $p$ be a prime $p \in \{24n + 11, 24n + 19 \}$.  For $m \geq 0$ denote by $c_m \in \mathbb{Q}$ the $m$-th coefficient in the expansion of ${}_2 F_{1}\left(\frac{11}{24},\frac{19}{24};\frac{3}{4};x\right)$. Then $c_m \equiv 0 \mod p$ for $n < m < 6n$. 
	\end{lem}
	\begin{proof}
		Similar to the proof of Lemma \ref{lem13}.
	\end{proof}

	For a prime $p$ congruent to $3$ modulo $4$, define the truncated hypergeometric function
	\begin{equation}
		G_p(\lambda) := {}_{2} F_{1}\left(-\frac{1}{4}, \frac{1}{4} ; \frac{1}{2};\lambda \right)_{\left(\frac{p+1}{4}\right)} = \sum_{m=0}^{\frac{p+1}{4}}\frac{(-\frac{1}{4})_m (\frac{1}{4})_m}{(\frac{1}{2})_m m!} \lambda^m,
	\end{equation}
	i.e. the hypergeometric function truncated at $\lambda^{\frac{p+1}{4}}$.
	\begin{lem}
		\label{lem25}
		Let $k = \frac{p+1}{2}$. Then the polynomial $P[\mathscr{C}_k \theta_{\z}]$ satisfies the transformation
		\[ \left( \frac{\lambda^2 (\lambda  - 1)^2}{256}\right)^{n_k} \cdot P[\mathscr{C}_k \theta_{\z}]\left( \frac{256(1 - \lambda + \lambda^2)^3}{\lambda^2(\lambda - 1)^2} \right) \equiv L_{a_k,b_k}(\lambda) \cdot
		G_p(\lambda) \mod p, \]
		where 
		\begin{align*}
			L_{a_k,b_k}(\lambda) = \frac{1}{(1 - \lambda + \lambda^2)^{a_k}(1 - \frac{3}{2} \lambda - \frac{3}{2} \lambda^2 + \lambda^3)^{b_k}}.\\
		\end{align*}
	\end{lem}
	\begin{proof}
		\emph{Cases $k \equiv 0, 4 \mod 12$.} In this case we have $b_k = 0$. Lemma \ref{lem13} implies 
		\[P[\mathscr{C}_k \theta_{\z}](j) \equiv j^{n_k} {}_2 F_{1}\left(-\frac{1}{24},\frac{7}{24};\frac{3}{4};\frac{1728}{j}\right) + \mathcal{O}(1/j^{5n_k}) \mod p.\]
		Applying the hypergeometric identity 
		\[ {}_2 F_{1}\left(-\frac{1}{24},\frac{7}{24};\frac{3}{4};\frac{27}{4} \frac{\lambda^2 (\lambda - 1)^2}{(1 - \lambda + \lambda^2)^3}\right) = (1 - \lambda + \lambda^2)^{-1/8} \cdot {}_2 F_{1}\left(-\frac{1}{4},\frac{1}{4};\frac{1}{2};\lambda \right),\] see \cite[Eq. (28)]{VIDUNAS},
		we find out that \begin{align*}
			L_{a_k,0}(\lambda)^{-1} \left( \frac{\lambda^2 (\lambda  - 1)^2}{256}\right)^{n_k} &\cdot P[\mathscr{C}_k \theta_{\z}]\left( \frac{256(1 - \lambda + \lambda^2)^3}{\lambda^2(\lambda - 1)^2} \right) \\ &\qquad \equiv (1 - \lambda + \lambda^2)^{3n_k + a_k - 1/8} \cdot {}_2 F_{1}\left(-\frac{1}{4},\frac{1}{4};\frac{1}{2};\lambda \right) + \mathcal{O}(\lambda^{12n_k}) \mod p
		\end{align*} as a power series in $\mathbb{F}_p[[\lambda]]$.
		Since 
		\[(1-\lambda + \lambda^2)^{3n_k + a_k - 1/8}= (1-\lambda + \lambda^2)^{p/8} \equiv 1 + \mathcal{O}(\lambda^p) \mod p\]
		and the left-hand side is a polynomial of degree $(p+1)/4$ in $\lambda$, we conclude that
		\[ \left( \frac{\lambda^2 (\lambda  - 1)^2}{256}\right)^{n_k} \cdot P[\mathscr{C}_k \theta_{\z}]\left( \frac{256(1 - \lambda + \lambda^2)^3}{\lambda^2(\lambda - 1)^2} \right) \equiv L_{a_k,0}(\lambda) \cdot
		G_p(\lambda) \mod p. \]
		\emph{Cases $k \equiv 6, 10 \mod 12$.} In this case we have $b_k = 1$.
		The identity
		\begin{align*}
			{}_{2} F_{1}\left(\frac{11}{24},\frac{19}{24};\frac{3}{4};\frac{1728}{j}\right) &= \left( 1 - \frac{1728}{j} \right)^{-1/2}  \cdot {}_{2} F_{1}\left(\frac{-1}{24},\frac{7}{24};\frac{3}{4};\frac{1728}{j}\right)\\
			&= \frac{(1 - \lambda + \lambda^2)^{3/2}}{1 - \frac{3}{2}\lambda - \frac{3}{2}\lambda^2 + \lambda^3} \cdot {}_{2} F_{1}\left(\frac{-1}{24},\frac{7}{24};\frac{3}{4};\frac{1728}{j}\right)
		\end{align*}
	gives, together with Lemma \ref{lem14},
		\begin{align*} L_{a_k,1}(\lambda)^{-1}  \left( \frac{\lambda^2 (\lambda  - 1)^2}{256}\right)^{n_k} &\cdot P[\mathscr{C}_k \theta_{\z}]\left( \frac{256(1 - \lambda + \lambda^2)^3}{\lambda^2(\lambda - 1)^2} \right) \mod p\\ & \qquad \equiv (1 - \lambda + \lambda^2)^{3n_k + a_k + 11/8} \cdot {}_2 F_{1}\left(-\frac{1}{4},\frac{1}{4};\frac{1}{2};\lambda \right) + \mathcal{O}(\lambda^{12n_k})\\
			&\qquad \equiv {}_2 F_{1}\left(-\frac{1}{4},\frac{1}{4};\frac{1}{2};\lambda \right) + \mathcal{O}(\lambda^{12n_k}).
		\end{align*}
		Again comparing the first $(p+1)/4$ coefficients on both sides we find out that
		\[ \left( \frac{\lambda^2 (\lambda  - 1)^2}{256}\right)^{n_k} \cdot P[\mathscr{C}_k \theta_{\z}]\left( \frac{256(1 - \lambda + \lambda^2)^3}{\lambda^2(\lambda - 1)^2} \right) \equiv L_{a_k,1}(\lambda) \cdot
		G_p(\lambda) \mod p. \] 
		This finishes the proof of the lemma.
	\end{proof}

	The goal is to show that the polynomial $G_p$ splits over $\mathbb{F}_p$; by Lemma \ref{lem25} this will imply that $P[\mathscr{C}_k \theta_{\z}](j)$ splits over $\mathbb{F}_p$.
	\begin{lem}
		\label{REC}
		The polynomial $G_p(\lambda)$ is reciprocal modulo $p$. That is,
		\[ \lambda^{\frac{p+1}{4}} G_p(1/\lambda)\equiv G_p(\lambda) \mod p. \]
	\end{lem}
	\begin{proof}
		Write $G_p(\lambda) = \sum_{j = 0}^{\frac{p+1}{4}} c_j \lambda^j.$ We need to show 
		\[c_{j} \equiv c_{\frac{p+1}{4} - j} \mod p \qquad \textnormal{ for } 0 \leq j \leq \frac{p+1}{4}. \]
		It is easy to see that $c_0 \equiv 1 \equiv c_{\frac{p+1}{4}} \mod p$. The congruence for the remaining coefficients follows by induction on $j$ and from the congruence
		\[ \frac{c_{j+1}}{c_j} \equiv \frac{(-\frac{1}{4}+j)(\frac{1}{4}+j)}{(\frac{1}{2}+j)(1+j)} \equiv \frac{(\frac{1}{2}+\frac{p+1}{4}-j-1)(\frac{p+1}{4}-j)}{(-\frac{1}{4}+\frac{p+1}{4}-j-1)(\frac{1}{4}+\frac{p+1}{4}-j-1)} \equiv \frac{c_{\frac{p+1}{4}-j-1}}{c_{\frac{p+1}{4}-j}} \mod p. \qedhere\]
	\end{proof}
	
	\begin{lem}
		\label{lem27}
		For a prime $p$ congruent to $3 \mod 4$, the polynomial $G_p$ factors as
		\begin{equation}
			\label{DezeMan}
			G_p(\lambda) \equiv \prod_{\substack{t - 1 \in \mathbb{F}_p^{*2}\\\ t \not \in \mathbb{F}_p^{*2}}} (\lambda - t) \mod p.
		\end{equation}
	\end{lem}
	\begin{proof}
		For a polynomial $P(x) \in \mathbb{F}_p[x]$ of degree $d>0$ and $v \geq 0$, define the $v$-th \emph{power sums} of $P$ as
		\[S_v(P) = \sum_{\alpha \colon P(\alpha) = 0} \alpha^v. \]
		Let $R_p(\lambda)$ be the polynomial on the right-hand side of \eqref{DezeMan}.
		The power sums can be written in terms of Legendre symbols as 
		\begin{align*}
			S_v(R_p) &\equiv \frac{1}{4} \sum_{a = 1}^{p-1} \left(1 - \left( \frac{a}{p} \right) \right) \left(1 + \left(\frac{a-1}{p} \right)\right) a^v \mod p\\
			&\equiv \frac{1}{4} \sum_{a = 1}^{p-1} \left(1 - a^{\frac{p-1}{2}} + (a-1)^{\frac{p-1}{2}} - a^{\frac{p-1}{2}} (a-1)^{\frac{p-1}{2}}  \right)a^v,
		\end{align*}
	where Euler's criterion is used to obtain the expression in the last line.
		Since
		\[    \sum_{a = 1}^{p-1} a^r \mod p \equiv \begin{cases}
			-1 &\,  \textnormal{ if } \, p-1 \text{ divides }r,\\
			\phantom{+}0  & \, \textnormal{ otherwise},
		\end{cases}
		\]
		it follows that 
		\[S_v(R_p) \equiv \frac{(-1)^v}{4} \binom{\frac{p-1}{2}}{v} \equiv \frac{1}{4} \frac{(\frac{1}{2})_v}{v!} \mod p. \]
		Using Lemma \ref{REC} and Newton's identities, we relate the coefficients of $G_p$ to the power sums of $G_p$ as follwos:
		\[S_v(G_p) \equiv -\sum_{j = 1}^{v-1} c_j S_{v-j}(G_p) - v c_v \mod p\]
	in the notation from the proof of Lemma \ref{REC}.
		Using induction on $v$ and the identity
		\[\sum_{j = 0}^v \frac{(-\frac{1}{4})_j (\frac{1}{4})_j}{(\frac{1}{2})_j j!} \frac{(\frac{1}{2})_{v-j}}{(v-j)!} = \frac{(\frac{1}{2})_{2v}}{(2v)!} = (1-4v) \frac{(-\frac{1}{4})_v (\frac{1}{4})_v}{(\frac{1}{2})_v v!}, \]
		we see that
		\[S_v(G_p) \equiv \frac{1}{4} \frac{(\frac{1}{2})_v}{v!} \equiv S_v(R_p) \mod p \]
		for all integers $0 \leq v \leq \frac{p+1}{4}$. As both polynomials have the same leading coefficient in $\mathbb{F}_p$, we conclude that $G_p(\lambda) \equiv R_p(\lambda) \mod p$.
	\end{proof}
	\begin{proof}[Proof of Theorem \ref{MainTHMB}.]
	It follows from Lemmas \ref{lem25} and \ref{lem27} that the zero set of the polynomials $P[\mathscr{C}_k \theta_{\z}](j) \mod p$, where $k = \frac{p+1}{2}$, is \[\left  \{ \frac{256(1 - \lambda + \lambda^2)^3}{\lambda^2 (\lambda - 1)^2} \, \Bigm |  -\lambda, \lambda - 1 \in \mathbb{F}_p^{*2} \right \} \setminus \{0, 1728 \}.\]
		Therefore, $P[\mathscr{C}_k \theta_{\z}](j)$ splits over $\mathbb{F}_p$.
	\end{proof}
	\subsection{Elliptic curves with a prescribed rational $4$-torsion group}
	The goal of this section is to classify all elliptic curves over $\mathbb{F}_p$ for $p \equiv 3 \mod 4$ with a prescribed rational $4$-torsion group. We will relate this to the zeros of $P[\mathscr{C}_k \theta_{\z}](j) \mod p$, where $k = \frac{p+1}{2}$.
	Every $\mathcal{E}/\mathbb{F}_p$ with rational $2$-torsion and $p \equiv 3 \mod 4$ is isomorphic (over $\mathbb{F}_p$) to a \emph{Legendre elliptic curve} $\mathcal{E}_{\lambda}/\mathbb{F}_p$ given by the equation
	\begin{equation*}
		\mathcal{E}_{\lambda}: Y^2 = X(X-1)(X-\lambda),
	\end{equation*}
	where $\lambda \in \mathbb{F}_p \setminus \{0, 1\}$. This follows from the explicit $\mathbb{F}_p$-isomorphism given in \cite[Proposition 1.7a]{Silverman}, see also \cite{TOP}.
	For these elliptic curves, we can classify all possible $\mathbb{F}_p$-rational 4-torsion groups. The next lemma is comparable with \cite[Proposition 2.1]{TOP}.
	\begin{lem}
		\label{Char}
		For $\lambda \in \mathbb{F}_p \setminus \{0,1\}$ and $p \equiv 3 \mod 4$,
		$$\mathcal{E}_{\lambda}(\mathbb{F}_p)[4] \cong \begin{cases} \mathbb{Z}/2\mathbb{Z} \times \mathbb{Z}/2\mathbb{Z} &\textnormal{ if } -\lambda, \lambda - 1  \in \mathbb{F}_p^{*2}, \\  
			\mathbb{Z}/2\mathbb{Z} \times \mathbb{Z}/4\mathbb{Z} &\textnormal{ otherwise.} \end{cases}$$
	\end{lem}
	\begin{proof}
		For the proof consider the division polynomial $\psi_4(X,Y) \in \mathbb{F}_p[X,Y]$, see \cite[p. 105]{Silverman}, of the elliptic curve $\mathcal{E}_\lambda: Y^2 = X(X-1)(X-\lambda)$. The zeros of this polynomial correspond precisely to points in $\mathcal{E}_{\lambda}(\overline{\mathbb{F}}_p)[4]$. A computation shows
		\begin{align*}
			\psi_4(X,Y) &\equiv 2Y (2X^6 - 4(1 + \lambda)X^5 + 10\lambda X^4 - 10\lambda^2X^2 + 4(1 + \lambda)\lambda^2X - 2\lambda^3) \mod p\\
			&\equiv 4Y(X^2 - \lambda)(X^2 - 2X + \lambda)(X^2 - 2\lambda X + \lambda).
		\end{align*}
		If $-\lambda, \lambda - 1 \in \mathbb{F}_p^{*2}$, we see that $\psi_4(X,Y)$ has no zeros in $\mathbb{F}_p$ apart from $Y = 0$. The remaining case uses a similar computation. For example, if  $\lambda, \lambda - 1 \in \mathbb{F}_p^{*2}$, we see that $$\mathcal{E}_\lambda(\mathbb{F}_p)[4] = \mathcal{E}_\lambda(\mathbb{F}_p)[2] \cup \langle (\lambda \pm \sqrt{\lambda(\lambda - 1)}, \lambda\sqrt{\lambda - 1} \pm \sqrt{\lambda}(\lambda - 1)) \rangle.$$
		Thus, $\mathcal{E}_\lambda(\mathbb{F}_p)[4] \cong \mathbb{Z}/2\mathbb{Z} \times \mathbb{Z}/4 \mathbb{Z}$ in this case.
	\end{proof}
	A combination of Lemmas \ref{Char} and \ref{lem27} gives the following result.
	\begin{lem}
		\label{lem28}
		For primes $p \equiv 3 \mod 4$, we have the congruence
		\[G_p(x) \equiv \prod_{\substack{\lambda \in \mathbb{F}_p \setminus \{0,1\}\\ \mathcal{E}_\lambda(\mathbb{F}_p)[4] \cong \mathbb{Z}/2\mathbb{Z} \times \mathbb{Z}/2\mathbb{Z}}} (x - \lambda) \mod p.\]
	\end{lem}
	\begin{proof}[Proof of Theorem \ref{MainTHMC}.] The zero set of $P[\mathscr{C}_k \theta_{\z}](j)$, where $k = \frac{p+1}{2}$, is precisely
		\[\left \{ \frac{256(1 - \lambda + \lambda^2)^3}{\lambda^2 (\lambda - 1)^2} \, \Bigm | \,  G_p(\lambda) \equiv 0 \mod p \right \} \setminus \{0, 1728 \}.\] As a consequence of Lemma \ref{lem28}, these are exactly the $j$-invariants, different from $0,1728$\,, of the elliptic curves over $\mathbb{F}_p$ with $\mathcal{E}(\mathbb{F}_p)[2] = \mathcal{E}(\mathbb{F}_p)[4] \cong \mathbb{Z}/2\mathbb{Z} \times \mathbb{Z}/2\mathbb{Z}$.
	\end{proof}
	
	\subsection{Proofs for $\mathscr{C}_k \theta_H(\tau)$}
		We start with a hypergeometric identity for the function $\theta_{H}(\tau)$.
	\begin{lem}
		In a neighborhood of $\tau = \ii \infty$, we have the identity
		\label{lem21}
		\begin{equation}
			\label{DEGE}
			\theta_{H}(\tau) = E_4(\tau)^{1/4} \, {}_2 F_{1} \left(-\frac{1}{12},\frac{1}{4};\frac{2}{3};\frac{1728}{j(\tau)}\right).
		\end{equation} 
	\end{lem}
	
	\begin{proof}
		Since $\theta_{H}(\tau)$ is a modular form of weight $1$ for the congruence subgroup $\Gamma_1(3)$, by \cite[Proposition 21]{Zagbook}, it satisfies a second order linear differential equation as a function in $1728/j$. The proof now follows the lines of the proof of Lemma \ref{lem22}.
	\end{proof}
	
	\begin{proof}[Proof of Theorem \ref{HexP}]
		This proof is similar to the proof of Theorem \ref{MainTHMA}.
		For weights $k \equiv 0, 6 \mod 12$ consider the modular form
		\begin{align*}
			g_k(\tau) = V^{b_k}_{n_k}(j(\tau)) \Delta(\tau)^{n_k} E_6(\tau)^{b_k}
		\end{align*}
	of weight $k$.
		If $k \equiv 0 \mod 12$, i.e. $b_k = 0$, then
		\begin{align*}
			g_k(\tau) = V^0_{n_k}(j(\tau)) \Delta(\tau)^{n_k} &=  \left(j^{n_k} \cdot {}_{2} F_{1}\left(-\frac{1}{12},\frac{1}{4};\frac{2}{3};\frac{1728}{j}\right) + \mathcal{O}(1/j)\right)\Delta^{n_k}\\
			&= \theta_{H} E_4^{3n_k - 1/4} + \mathcal{O}(q^{n_k + 1})
		\end{align*}
		by Lemma \ref{lem21}. As
		\[ E_4^{3n_k - 1/4} = E_4^{p/4} \equiv 1 + \mathcal{O}(q^{p}) \mod p,\]
		it follows that $g_k \equiv \mathscr{C}_k \theta_H + \mathcal{O}(q^{n_k + 1}) \mod p$ and, therefore, $g_k \equiv \mathscr{C}_k \theta_H \mod p$ by Lemma \ref{almostzeroiszero}.

		For the case $k \equiv 6 \mod 12$, i.e. $b_k = 1$, make use of Euler's transformation
		\begin{align}
			{}_{2} F_{1}\left(\frac{5}{12},\frac{3}{4};\frac{2}{3};\frac{1728}{j}\right) &= 
			\left(1 - \frac{1728}{j} \right)^{-1/2} \, {}_{2} F_{1}\left(-\frac{1}{12},\frac{1}{4};\frac{2}{3};\frac{1728}{j}\right)\label{Euler}\\
			&= E_4^{3/2} E_6^{-1} \, {}_{2} F_{1}\left(-\frac{1}{12},\frac{1}{4};\frac{2}{3};\frac{1728}{j}\right). \nonumber
		\end{align}
		Similar to the case $k \equiv 0 \mod 12$ we find out that
		\begin{align*}
			g_k = \theta_{H} E_4^{3n_k + 5/4}+ \mathcal{O}(q^{n_k + 1}).
		\end{align*}
		As
		\begin{align*}
			E_4^{3n_k + 5/4} = E_4^{p/4} \equiv 1 + \mathcal{O}(q^{p}) \mod p,
		\end{align*}
	using Lemma \ref{almostzeroiszero} 	we conclude that $g_k(\tau) \equiv \mathscr{C}_k \theta_H(\tau) \mod p$.
		Finally, this shows that $V^{b_k}_{n_k}(j) \equiv P[\mathscr{C}_k \theta_H](j) \mod p$.
	\end{proof}

As in the case of $\mathscr{C}_k \theta_{\z}(\tau)$, we next choose a Hauptmodul for the group $\Gamma_1(3)$. Here we pick
\[t_3(\tau) = -108 \frac{\left(\frac{\eta(3 \tau)}{\eta(\tau)}\right)^{12}}{1 + 27 \left(\frac{\eta(3 \tau)}{\eta(\tau)}\right)^{12}} = -108 q + 1620 q^2 - 18468 q^3 + 181332 q^4 - 1625832 q^5 + \cdots.\]
Indeed, we can express the $j$-invariant as a rational function in $t_3(\tau)$:
\[j(\tau) = \frac{3^3 4^4 (2 t_3(\tau) - 1)^3}{t_3(\tau)(t_3(\tau)+4)^3}, \]
see \cite[Theorem 4.32]{Cooper}. This transformation is used in Lemma \ref{lem213} below.
We first establish analogues of Lemmas \ref{lem13} and \ref{lem14}.
	
	\begin{lem}
		\label{lem23}
		For $n \in \mathbb{Z}_{\geq 0}$ let $p$ be a prime $p = 12n - 1$.  For $m \geq 0$ denote by $c_m \in \mathbb{Q}$ the $m$-th coefficient in the expansion of ${}_2 F_{1}\left(-\frac{1}{12},\frac{1}{4};\frac{2}{3};x\right)$. Then $c_m \equiv 0 \mod p$ for $n < m < 4n$. 
	\end{lem}
	\begin{proof}
		This follows from considering the $p$-adic valuation of the coefficients.
		We have $\nu _p ( (-\frac{1}{12})_m) \geq 1$ if and only if $m > n$ and $\nu _p ( (
		\frac{2}{3})_m) \geq 1$ if and only if $m \geq 4n$.
	\end{proof}
	
	\begin{lem}
		\label{lem24}
		For $n \in \mathbb{Z}_{\geq 0}$ let $p$ be a prime $p = 12n + 5$.  For $m \geq 0$ denote by $c_m \in \mathbb{Q}$ the $m$-th coefficient in the expansion of ${}_2 F_{1}\left(\frac{5}{12},\frac{3}{4};\frac{2}{3};x\right)$. Then $c_m \equiv 0 \mod p$ for $n < m < 4n + 2$. 
	\end{lem}
	\begin{proof}
		Similar to the proof of Lemma \ref{lem23}.
	\end{proof}
	\begin{lem}
		\label{lem213}
		Let $k = p + 1$. Then the polynomial $P[\mathscr{C}_k \theta_H](j)$ satisfies the transformation
		\[ (y (y+4)^3)^{n_k}P[\mathscr{C}_k \theta_H]\left(\frac{3^3 4^4 (2y-1)^3}{y(y+4)^3}\right) \equiv
		L_{b_k}(y) \cdot \bigl(y^{(p+1)/3} + 2^{1/3} \bigr)  \mod p, \]
		where $2^{1/3}$ is interpreted as the unique cube root of $2$ in $\mathbb{F}_p$, 
		\begin{align*}
			L_0(y) \equiv 12^{(p+1)/4} \qquad \textnormal{ and } \qquad
			L_1(y) \equiv \frac{12^{(p-5)/4}}{y^2 - 10y - 2}.
		\end{align*}
	\end{lem}
	\begin{proof}
	The existence and uniqueness of $2^{1/3} \in \mathbb{F}_p$ is guaranteed by $p \equiv 5, 11 \mod 12$.
	
		\emph{Case $k \equiv 0 \mod 12$.} Using Lemma \ref{lem21} we obtain
		\[ P[\mathscr{C}_k \theta_H](j) \equiv j^{n_k} {}_2 F_{1} \left(-\frac{1}{12},\frac{1}{4};\frac{2}{3};\frac{1728}{j}\right) - \frac{c}{j^{3 n_k}} + \mathcal{O}(1/j^{3 n_k + 1}) \mod p,\] which is a stronger version of Theorem \ref{HexP}. Here $c$ is the quantity
		\begin{align*}
			c \equiv \frac{(-\frac{1}{12})_m (\frac{1}{4})_m}{(\frac{2}{3})_m \, m!} 1728^m\,  \Big |_{m = (p+1)/3} \equiv - 18 \mod p.
		\end{align*}
		Applying the hypergeometric transformation
		\[  {}_2 F_{1} \left(-\frac{1}{12},\frac{1}{4};\frac{2}{3};\frac{y(y+4)^3}{4 (2y - 1)^3}\right) = (1 -2y)^{-1/4},\]
		see \cite[Eq. (2.1)]{VIDUNAS},
		we find out that
		\begin{align*}
			(y (y+4)^3)^{n_k}P[\mathscr{C}_k \theta_H]\left(\frac{3^3 4^4 (2y-1)^3}{y(y+4)^3}\right) &\equiv (-3^3 4^4)^{(p+1)/12} (1-2y)^{(p+1)/4}(1-2y)^{-1/4}\\ & \qquad- c \, (-3^{-3})^{(p+1)/4} y^{(p+1)/3} + \mathcal{O}(y^{(p+1)/3 + 1}) \mod p\\
			&\equiv 12^{(p+1)/4} y^{(p+1)/3} + (-3^3 4^4)^{(p+1)/12}\\ &\equiv 12^{(p+1)/4} \bigl(y^{(p+1)/3} + (-4)^{(p+1)/12}\bigr)
		\end{align*} as a power series in $\mathbb{F}_p[[y]]$.
	Here the $\mathcal{O}$-term is dropped since the left-hand side is a polynomial of degree (at most) $(p+1)/3$ in $y$. It remains to notice that $(-4)^{(p+1)/12} \equiv 2^{1/3} \mod p$.		
		
		\emph{Case $k \equiv 6 \mod 12$.} Lemma \ref{lem21} implies
		\begin{equation} 
			\label{Phex}
			P[\mathscr{C}_k \theta_H](j) \equiv j^{n_k} {}_2 F_{1} \left(\frac{5}{12},\frac{3}{4};\frac{2}{3};\frac{1728}{j}\right) - \frac{c'}{j^{3 n_k + 2}} + \mathcal{O}(1/j^{3 n_k + 3}) \mod p,\end{equation}
		which is again a stronger version of Theorem \ref{HexP}. Here $c'$ is the quantity \[c' \equiv \frac{(\frac{5}{12})_m (\frac{3}{4})_m}{(\frac{2}{3})_m \, m!} 1728^m\,  \Big|_{m = (p+1)/3} \equiv - 18 \mod p.\]
		From the transformation \eqref{Euler} we learn that
		\begin{align}
			{}_{2} F_{1}\left(\frac{5}{12},\frac{3}{4};\frac{2}{3};\frac{y(y+4)^3}{4 (2y-1)^3}\right) &= 
			 \left(1 - \frac{y(y+4)^3}{4 (2y -1)^3} \right)^{-1/2} \cdot {}_{2} F_{1}\left(-\frac{1}{12},\frac{1}{4};\frac{2}{3};\frac{y(y+4)^3}{4 (2y-1)^3}\right)\nonumber\\
			 &=-\frac{2 (1-2y)^{3/2}}{y^2-10y-2} \cdot (1-2y)^{-1/4} = -\frac{2 (1-2y)^{5/4}}{y^2-10y-2}. \label{Phex2}
		\end{align}
	
As a consequence of \eqref{Phex} and \eqref{Phex2} we find out that
		\begin{align*}
			& (y^2 - 10y- 2) (y (y+4)^3)^{n_k}P[\mathscr{C}_k \theta_H]\left(\frac{3^3 4^4 (2y-1)^3}{y(y+4)^3}\right) \\&\qquad \equiv -2 (-3^3 4^4)^{(p-5)/12} (1 - 2y)^{p/4}+\frac{c'}{8}(-3^{-3})^{(p+3)/4}y^{(p+1)/3}+ \mathcal{O}(y^{(p+1)/3+1}) \mod p\\
			&\qquad\equiv 12^{(p-5)/4} y^{(p+1)/3} - 2(-3^3 4^4)^{(p-5)/12} \\&\qquad\equiv 12^{(p-5)/4}\bigl(y^{(p+1)/3} -2 (-4)^{(p-5)/12}\bigr). 
		\end{align*}
It remains to notice that $-2(-4)^{(p-5)/12} \equiv 2^{1/3} \mod p$.
	\end{proof}
	\begin{proof}[Proof of Theorem \ref{HexF}] 
		Since $(p+1)/3$ is a divisor of $p^2-1$, the polynomial $y^{(p+1)/3} + 2^{1/3}$ divides $y^{p^2-1} - 1$. The latter polynomial splits over $\mathbb{F}_{p^2}$.
		More precisely, the zero set of $P[\mathscr{C}_k \theta_H](j) \mod p$ is
		\begin{equation}
			\label{hexazeros}
			\left\{ \frac{3^3 4^4 (2 a -1)^3}{a(a+4)^3} \, \Bigm | \,  a^{(p+1)/3} + 2^{1/3} = 0, \, a \in \mathbb{F}_{p^2}  \right\}  \setminus \{0, 1728\}.\end{equation}
We next show that $P[\mathscr{C}_k \theta_H](j) \mod p$ has at most one zero in $\mathbb{F}_p$. Suppose $\beta = \frac{3^3 4^4 (2a -1)^3}{a(a+4)^3} \in \mathbb{F}_{p^2}$ is a zero of $P[\mathscr{C}_k \theta_H](j)$.  From the relation $a^p = -2/a$ we see that
\begin{align*}
	\beta^p = \frac{3^3 4^4 (2 a^p - 1)^3}{a^p (a^p + 4)^3} &= \frac{3^3 4^4 (\frac{-4}{a} - 1)^3}{\frac{-2}{a} (\frac{-2}{a} + 4)^3} = \frac{1728^2}{\beta}.
\end{align*}
Since $\beta \in \mathbb{F}_p$ if and only if $\beta^p = \beta$, it is clear that $\beta = -1728$. Furthermore, by comparing the degree of $P[\mathscr{C}_k \theta_H](j)$ with the cardinality of \eqref{hexazeros}, we find that $\beta = -1728$ is a simple zero of  $P[\mathscr{C}_k \theta_H](j) \mod p$.  Therefore, $P[\mathscr{C}_k \theta_H](j)$ factors as a product of quadratic factors if $n_k$ is even and as $(j + 1728)$ times quadratic factors if $n_k$ is odd. \qedhere
	
	\end{proof}

\paragraph{Acknowledgements}

This paper grew out of the author's master thesis written at Utrecht University and was further developed at the Radboud University Nijmegen. I thank both institutions for wonderful working conditions.
The author would like to thank Gunther Cornelissen, Mar Curcó Iranzo, David Hokken and Wadim Zudilin for their valuable comments.


	\bibliographystyle{amsplain}
	\bibliography{Berendsbronnen2}
	
\end{document}